\documentclass[12pt]{amsart}
\usepackage{amsmath}
\usepackage{graphicx}
\usepackage{hyperref}
\usepackage[latin1]{inputenc}
\usepackage{mathtools}
\usepackage{amsfonts}
\usepackage{amssymb}
\usepackage[T1]{fontenc}
\usepackage{amsthm}
\usepackage{fullpage}
\usepackage{enumitem}
\usepackage{longtable}

\newtheorem{theorem}{Theorem}[section]

\newtheorem{lemma}[theorem]{Lemma}

\theoremstyle{definition}
\newtheorem{definition}{Definition}[section]

\theoremstyle{remark}

\numberwithin{equation}{section}

\newcommand{\Fm}{\mathbb{F}_{q^m}}

\title{Primitive normal pairs with prescribed traces over finite fields}


\author{Shikhamoni Nath}
\address{Department of Mathematical Sciences, Tezpur University, Tezpur, Assam, 784028, India}
\email{shikha@tezu.ernet.in}

\author{Arpan Chandra Mazumder}
\address{Department of Mathematical Sciences, Tezpur University, Tezpur, Assam, 784028, India}
\email{arpan10@tezu.ernet.in}

\author{Dhiren Kumar Basnet}
\address{Department of Mathematical Sciences, Tezpur University, Tezpur, Assam, 784028, India}
\email{dbasnet@tezu.ernet.in}

\begin{document}
	%\vspace{.3cm}
	\begin{abstract}
		Let $q$ be a positive integral power of some prime $p$ and $\mathbb{F}_{q^m}$ be a finite field with $q^m$ elements for some $m \in \mathbb{N}$. Here we establish a sufficient condition for the existence of primitive normal pairs of the type $(\epsilon, f(\epsilon))$ in $\mathbb{F}_{q^m}$ over $\mathbb{F}_{q}$ with two prescribed traces, $Tr_{{\mathbb{F}_{q^m}}/{\mathbb{F}_q}}(\epsilon)=a$ and $Tr_{{\mathbb{F}_{q^m}}/{\mathbb{F}_q}}(f(\epsilon))=b$, where $f(x) \in \mathbb{F}_{q^m}(x)$ is a rational function with some restrictions and $a, b \in \mathbb{F}_q$. Furthermore, for $q=5^k$, $m \geq 9$ and rational functions with degree sum 4, we explicitly find at most 12 fields in which the desired pair may not exist.
	\end{abstract}
	
	\maketitle
	
	\section{Introduction}
	
	Given a prime power  $q$  and a positive integer $m$, we denote the finite field with $q$ elements by $\mathbb{F}_q$ and its extension field of degree $m$ by $\mathbb{F}_{q^m}$. The multiplicative group $\mathbb{F}^*_{q^m}$ is a cylic group and a generator of this group is called a {\em primitive} element. For an element $\epsilon\in\mathbb{F}_{q^m}$, the elements $\epsilon,\epsilon^q, \epsilon^{q^2},\ldots, \epsilon^{q^{m-1}}$ are called conjugates of $\epsilon$ with respect to $\mathbb{F}_q$. Sum of all these conjugates is called trace of $\epsilon$ over $\mathbb{F}_q$ and is denoted by $Tr_{\mathbb{F}_{q^m}/\mathbb{F}_q}(\epsilon)$. Further, an element $\epsilon\in\mathbb{F}_{q^m}$ for which the set $\{\epsilon,\epsilon^q, \epsilon^{q^2},\ldots, \epsilon^{q^{m-1}}\}$ forms a basis of $\mathbb{F}_{q^m}( \mathbb{F}_q) $, when considered as a vector space, is called a $normal$ element. An element which is simultaneously primitive and normal is referred as a primitive normal element. Existence of such elements is firstly given by the following theorem.
	
	\begin{theorem}{\textbf{ \em (Primitive normal basis theorem)}}. In the finite field $\mathbb{F}_{q^m}$, there always exists some element which is simultaneously primitive and normal.
	\end{theorem}
	
	Lenstra and Schoof proved this result in \cite{2}. Later, Cohen and Huczynska \cite{3} presented a computer-free proof of this result with the help of sieving techniques.
	
	In 1985, Cohen \cite{pair}, proved existence of consecutive pairs of primitive elements of the type $(\epsilon, \epsilon^{-1})$ in $\mathbb{F}_q$. After that, it was further developed in \cite{cohen1}, \cite{booker}, \cite{cohen2}, among others.
	Similar studies on primitive normal pairs emerged following the development of the primitive normal basis theorem; they can be studied from \cite{carvalho}, \cite{kapetanakis}, \cite{sharma}, and so on, as well as the references therein.
	\par In 2001, Chou and Cohen \cite{chou} studied existence of primitive pair $(\epsilon, \epsilon^{-1})$ in $\mathbb{F}_q$ such that both have trace zero over $\mathbb{F}_p$. Cao and Wang \cite{cao} proved that, for any prime power $q$ and any positive integer $m$, there exists elements $\epsilon \in \mathbb{F}_{q^m}$ such that $\epsilon+\epsilon^{-1}$ is a primitive element of $\mathbb{F}_{q^m}$ and $Tr_{{\mathbb{F}_{q^m}}/{\mathbb{F}_q}}(\epsilon)=a$ and $Tr_{{\mathbb{F}_{q^m}}/{\mathbb{F}_q}}(\epsilon^{-1})=b$ for any prescribed $a, b \in \mathbb{F}^*_{q}$. In \cite{choudhary}, Choudhary et al. discussed the existence of primitive pairs $(\epsilon, f(\epsilon))$ with two prescribed traces, where $f$ is a rational function with some restrictions. Recently Sharma et al. \cite{sharma2} considered rational functions $f$ with some minor restrictions over $\mathbb{F}_{q^m}$ for some prime power $q$ and $m \in \mathbb{N}$ and provided a sufficient condition for the existence of primitive normal pair of the form $(\epsilon, f(\epsilon))$ over $\mathbb{F}_q$.
	
	Inspired by all these results we are presenting a sufficient condition for existence of primitive normal pair $(\epsilon,f(\epsilon))$ in $\mathbb{F}_{q^m}$ over $\mathbb{F}_q$ such that $Tr_{\mathbb{F}_{q^m}/\mathbb{F}_q}(\epsilon)=a$ and $Tr_{\mathbb{F}_{q^m}/\mathbb{F}_q}(f(\epsilon))=b$ where $f(x) \in \mathbb{F}_{q^m}(x)$ is a rational function with some restrictions and $a,b \in \mathbb{F}_q$. The kind of rational functions we will study in this paper is defined below:
	
	\begin{definition}\label{def1.1}
		The set $\mathcal{R}_{q^m}^n$ consists of rational functions of the simplest form $f=\frac{f_1}{f_2}$ of degree sum $n=\deg(f_1)+\deg(f_2)$, such that \\
		
		(i) $f \neq cx^{j}h^{d}$ for any $c \in \mathbb{F}^*_{q^m}$, any $j \in \mathbb{Z}$, any $h\in \mathbb{F}_{q^m}(x)$ and any divisor $d$ other than $1$ of $q^m-1$.\\
		
		(ii) There exists at least one monic irreducible factor $g$ of $f_2$ in $\mathbb{F}_{q^m}[x]$ with multiplicity $r$ such that $q^m \nmid r$.  
	\end{definition}
	We will use $\deg(f_1)=n_1$ and $\deg(f_2)=n_2$ throughout.
	We also note that the second condition of the above definition implies that $n_2>0$. This type of rational functions are popularly known as Cohen's kind of rational functions \cite{cohen2} that also includes Carvalho's kind of rational functions \cite{carvalho}.
	
	\par In Section 3, we derive a sufficient condition for the existence of desired primitive normal pairs. In Section 4, we use a sieving inequality to weaken the sufficient condition for more efficient results. In Section 5, we study the application of our results over fields of chracteristic $5$ and for rational functions of degree sum $4$ and proved the following:
	
	\begin{theorem}\label{main}
		Let $q=5^k, m\geq9$ and $f \in \mathcal{R}_{q^m}^n$. Then for any $a,b \in \mathbb{F}_q$, there exists an element $\epsilon \in \mathbb{F}_{q^m}$ with $Tr_{{\mathbb{F}_{q^m}}/{\mathbb{F}_q}}(\epsilon)=a$, $Tr_{{\mathbb{F}_{q^m}}/{\mathbb{F}_q}}(f(\epsilon))=b$ and $(\epsilon,f(\epsilon))$ is a primitive normal pair in $\mathbb{F}_{q^m}$ over $\mathbb{F}_q$ unless $(q,m)$ is one of the $12$ pairs : $ (5,9)$, $(5,10)$, $(5,11)$, $(5,12)$, $(5,14)$, $(5,18)$, $(5,20)$, $(5,24)$, $(5^2,9)$, $(5^2,10)$, $(5^2,12)$, $(5^2,24)$. 
	\end{theorem}

	\par Now, we present some notations that are significant to this article. We indicate the algebraic closure of $\mathbb{F}_{q^m}$ by $\mathbb{F}$. Let $\mathcal{A}_p^n$ be the set of pairs $(q,m)$ such that there exists an element $\epsilon \in \mathbb{F}_{q^m}$ with  $Tr_{{\mathbb{F}_{q^m}}/{\mathbb{F}_q}}(\epsilon)=a$ and $Tr_{{\mathbb{F}_{q^m}}/{\mathbb{F}_q}}(f(\epsilon))=b$, where $f(x) \in \mathcal{R}_{q^m}^n$ and $a, b \in \mathbb{F}_q$.
	For each $n \in \mathbb{N}$, we denote the number of prime divisors of $n$ and the number of square-free divisors of $n$  by $w(n)$ and $W(n)$, respectively. Additionally, for $r(x) \in \mathbb{F}_q[x]$, we indicate the number of square-free and monic irreducible $\mathbb{F}_q$-divisors of $r$ by $W(r)$ and $w(r)$, respectively.

	\section{Preliminaries}
	
	In this section, we review a few definitions, lemmas, and notations that will be used in proving a sufficient condition for the existence of primitive normal pairs of our interest.
	\begin{definition}{\textbf{(Characters)}}
		Let $G$ be a finite abelian group. Then a character $\chi$ of $G$ is a homomorphism from $G$ to the group $S^1:= \{z \in \mathbb{C} : |z| = 1\}$. The characters of $G$ form a group under multiplication called $\mathit{dual\thinspace group}$ or $\mathit{character\thinspace group}$ of $G$ which is denoted by $\widehat{G}$. It is well known that $\widehat{G}$ is isomorphic to $G$. Also, $\chi_1$ denotes the trivial character of $G$ defined as $\chi_1(a)=1_G$, for all $a \in G$, where $1_G$ is the identity element of the group.
		\par  If the order of a character of the group $G$ is $d$, then it is denoted by $\chi_d$. We also denote all the characters of order $d$ by $(d)$. 
		
	\end{definition}
	
	Corresponding to two different abelian group structures present in a finite field $\mathbb{F}_{q^m}$ i.e., $(\mathbb{F}_{q^m}, +)$ and $(\mathbb{F}^*_{q^m},\cdot)$; two types of characters can be formed. We refer to them as additive character denoted by $\psi$ and multiplicative character denoted by $\chi$ respectively. Multiplicative characters are extended from $\mathbb{F}^*_{q^m}$ to $\mathbb{F}_{q^m}$ by the  rule
	\hspace{.1cm}  $\chi(0)=\begin{cases}
		0 \,\mbox{ if}\, \chi\neq\chi_1,\\
		1 \,\mbox{ if}\, \chi=\chi_1.
	\end{cases} $
	
	The additive character $\bar{\psi}(\epsilon)=e^{2{\pi}iTr(\epsilon)/p}$, for all $\epsilon \in \mathbb{F}_{q}$, where $Tr$ is the absolute trace function from $\mathbb{F}_{q}$ to $\mathbb{F}_p$, is called the canonical additive character of $\mathbb{F}_{q}$ and every additive character $\psi_{\alpha}$ for $\alpha \in \mathbb{F}_{q}$ can be expressed in terms of the cannonical additive character $\bar{\psi}$ as $\psi_{\alpha}(\beta)=\bar{\psi}(\alpha\beta)$ for all $\beta \in \mathbb{F}_{q}$.

	\begin{definition}{\textbf{($e$-free element)}}
		For any divisor $e$ of $q^m-1$, an element $\epsilon \in \mathbb{F}_{q^m}$ is called $e$-free if $\epsilon \neq \beta^d$ for any $\beta \in \mathbb{F}_{q^m}$ and for any divisor $d$ of $e$ other than $1$.   
	\end{definition}
	
	An element $\epsilon \in \mathbb{F}^*_{q^m}$ is primitive if and only if $(q^m-1)$-free. The characteristic function for $e$-free elements is defined as follows:
	
	\begin{equation}\label{cfe}
		\rho_e: \mathbb{F}^*_{q^m} \to \{ 0,1\};   \epsilon \mapsto \theta(e) \sum_{d|e}\frac{\mu(d)}{\phi(d)}\sum_{(d)}\chi_{d}(\epsilon),
	\end{equation}
	where $\theta(e)=\frac{\phi(e)}{e}$ and $\phi, \mu$ denotes the Euler's totient function and Mobius function respectively.

	The additive group of $\mathbb{F}_{q^m}$ is a $\mathbb{F}_{q}[x]$-module under the rule $f \circ\alpha=\overset{n}{\underset{i=0}{\sum}} a_i\alpha^{q^i}$; for $\alpha\in \mathbb{F}_{q^m}$ and
	$f\,= \overset{n}{\underset{i=0}{\sum}}a_ix^i\thinspace \in \mathbb{F}_{q}[x]$. For $\alpha \in \mathbb{F}_{q^m}$, the $\mathbb{F}_q$-order of $\alpha$ is the monic $\mathbb{F}_q$-divisor $g$ of $x^m-1$ of minimal degree such that $g \circ\alpha=0$.
	
	Similarly we can define order of an additive character. Let $\psi \in \widehat{\mathbb{F}_{q^m}}$. Then the $\mathbb{F}_q$-order of $\psi$ is the least monic divisor $g$ of $x^m-1$ such that $\psi \circ g $ is the trivial character where $\psi \circ g (\beta)=\psi(g\circ \beta)$.
	
	\begin{definition}{\textbf{($g$-free element)}}
		Let $g$ divides $x^m-1$ and $\epsilon \in \mathbb{F}_{q^m}$. If $\epsilon \neq h \circ \gamma$ for any $\gamma \in \mathbb{F}_{q^m}$ and for any divisor $h$ of $g$ other than $1$, then $\epsilon$ is called a $g$-free element.
	\end{definition} 
	An element $\epsilon \in \mathbb{F}_{q^m}$ is normal if and only if it is $(x^m-1)$-free.
	The characteristic function for $g$-free elements is defined as follows:
	\begin{equation}\label{cfg}
		\eta_g: \mathbb{F}_{q^m} \to \{ 0,1\};   \epsilon \mapsto \Theta(g) \sum_{h|g}\frac{\mu^{'}(g)}{\Phi(g)}\sum_{(h)}\psi_{h}(\epsilon),
	\end{equation}
	
	where $\Theta(g)= \frac{\Phi(g)}{q^{\deg(g)}}$ and $\Phi(g)=|{(\frac{\mathbb{F}_{q}[x]}{<g>})}^{*}|$ and the function $\mu^{'}$ is defined as follows:
	
	$$\mu^{'}(g)=\begin{cases}
		(-1)^s,  \text{ when $g$ is a product of $s$ distinct monic irreducible polynomials}\\
		0, \text{ otherwise}.
	\end{cases} $$
	
	We shall also need characteristic function for prescribed values of trace. For each $a \in \mathbb{F}_q$ the characteristic function for elements $\epsilon \in \mathbb{F}_{q^m}$ such that $Tr_{{\mathbb{F}_{q^m}}/{\mathbb{F}_q}}(\epsilon)=a$ is defined as follows:
	
	\begin{equation}\label{cft}
		\tau_a: \mathbb{F}_{q^m} \to \{ 0,1\}; \epsilon \mapsto \frac{1}{q}\sum_{\psi \in \widehat{\mathbb{F}_q}}\psi(Tr_{{\mathbb{F}_{q^m}}/{\mathbb{F}_q}}(\epsilon)-a).
	\end{equation}
	
	The following lemmas will be handy while finding the sufficient conditions.
	
		\begin{lemma} \label{lemma1}\cite{todd}
		Let $f(x)\in \mathbb{F}_{q^m}(x)$ be a rational function. Write $f(x)=\overset{k}{\underset{j=1}{\prod}}{f_j(x)}^{r_j}$, where $f_j(x) \in \mathbb{F}_{q^m}[x]$ are irreducible polynomials and $r_j $ are nonzero integers. Let $\chi$ be a nontrivial multiplicative character of $\mathbb{F}_{q^m}$ of squarefree order $d$ (a divisor of $q^m-1$). Suppose that $f(x)$ is not of the form $cg(x)^d$ for any rational function $g(x)\in \mathbb{F}(x)$ and $c \in \mathbb{F}^*_{q^m}$. Then we have,
		
		$$\bigg|{\underset{\alpha\in \mathbb{F}_{q^m},f(\alpha)\neq 0, \infty}{\sum}\chi(f(\alpha))}\bigg| \leq \bigg(\overset{k}{\underset{j=1}{\sum}}\deg(f_j)-1\bigg)q^{\frac{m}{2}}.$$
	\end{lemma}
	
	\begin{lemma}\label{lemma2} \cite{castro}
		Let $\chi$ and $\psi$ be two non-trivial multiplicative and additive characters of the field $\Fm$ respectively. Let $f,g$ be rational functions in $\Fm (x)$, where $f\neq \beta h^r$ and $g \neq h^p-h+\beta$, for any $h\in \Fm(x)$ and any $\beta\in \Fm$, and $r$ is the order of $\chi$. 
		Then
		$$\left | \underset{\alpha\in\mathbb{F}_{q^m} \setminus S}{\sum} \chi(f(\alpha))\psi(g(\alpha))\right|\leq [\deg(g_\infty)\, +\,  l_0 \,+\, l_1\,-\,l_2\,-\,2 ]q^{m/2},$$ 
		where $S$ denotes the set of all poles of $f$ and $g$, $g_\infty$ denotes the pole divisor of $g$, $l_0$ denotes  the number of distinct zeroes and poles of $f$ in the algebraic closure $\overline{\Fm}$ of $\Fm$, $l_1$ denotes the number of distinct poles of $g$ (including infinite pole) and $l_2$ denotes the number of finite poles of $f$, that are also zeroes or poles of $g$.
	\end{lemma}

	Using these, we can now proceed to  find all the pairs $(q,m)$ for which there exists a rational function $f \in \mathcal{R}_{q^m}^n$ of degree sum $n$, $(\epsilon, f(\epsilon))$ is a primitive normal pair with $Tr_{{\mathbb{F}_{q^m}}/{\mathbb{F}_q}}(\epsilon)=a$ and $Tr_{{\mathbb{F}_{q^m}}/{\mathbb{F}_q}}(f(\epsilon))=b$ for some prescribed values $a,b  \in \mathbb{F}_q$.

	\section{A sufficient condition for elements in $\mathcal{A}_p^n$}
	Let $e_1,e_2$ be two divisors of $q^m-1$ and $g_1,g_2$ be two divisors of $x^m-1$. We denote the number of $\epsilon \in \mathbb{F}_{q^m}$ such that $\epsilon$ is both $e_1$-free and $g_1$-free and  $f(\epsilon)$ is both $e_2$-free and $g_2$-free with   
	$Tr_{{\mathbb{F}_{q^m}}/{\mathbb{F}_q}}(\epsilon)=a$ and $Tr_{{\mathbb{F}_{q^m}}/{\mathbb{F}_q}}(f(\epsilon)=b$, where $a,b\in \mathbb{F}_q$ and $f(x) \in \mathcal{R}_{q^m}^n$ by $N_{f,a,b}(e_1,e_2,g_1,g_2)$. Let $S$ be the set containing all zeros and poles of $f(x)$ along with $0$.
	Now using the characteristic functions (\ref{cfe}), (\ref{cfg}), and (\ref{cft}),  we get 
	$$N_{f,a,b}(e_1,e_2,g_1,g_2)=\underset{\epsilon\in \mathbb{F}_{q^m}\textbackslash S}{\sum} \rho_{e_1}(\epsilon)\rho_{e_2}(f(\epsilon))\eta_{g_1}(\epsilon)\eta_{g_2}(f(\epsilon))\tau_a(\epsilon)\tau_b(f(\epsilon))$$
	
	\begin{align*}
		N_{f,a,b}(e_1,e_2,g_1,g_2) =& \underset{\epsilon\in \mathbb{F}_{q^m}\textbackslash S}{\sum} \rho_{e_1}(\epsilon)\rho_{e_2}(f(\epsilon))\eta_{g_1}(\epsilon)\eta_{g_2}(f(\epsilon)) \tau_a(\epsilon)\tau_b(f(\epsilon))\\
		=& \frac{\theta(e_1)\theta(e_2)\Theta(g_1)\Theta(g_2)}{q^2} \underset{\substack{d_1|e_1,d_2|e_2 \\ h_1|g_1,h_2|g_2}}{\sum}\frac{\mu(d_1)\mu(d_2)\mu'(h_1)\mu'(h_2)}{\phi(d_1)\phi(d_2)\Phi(h_1)\Phi(h_2)} \\
		& \underset{\substack{\chi_{d_1},\chi_{d_2} \\ \psi_{h_1},\psi_{h_2}}}{\sum}\chi_{f,a,b}(d_1,d_2,h_1,h_2)
	\end{align*}

	where, 
	\begin{align*}
		\chi_{f,a,b}(d_1,d_2,h_1,h_2) =& \underset{\epsilon\in \mathbb{F}_{q^m}\textbackslash S}{\sum}\chi_{d_1}(\epsilon)\chi_{d_2}(f(\epsilon))\psi_{h_1}(\epsilon)\psi_{h_2}(f(\epsilon))\underset{\psi \in \widehat{\mathbb{F}_q}}{\sum}\psi(Tr(\epsilon)-a)\\
		& \underset{\psi' \in \widehat{\mathbb{F}_q}}{\sum}\psi'(Tr(f(\epsilon))-b)\\
		=& \underset{\epsilon\in \mathbb{F}_{q^m}\textbackslash S}{\sum}\chi_{d_1}(\epsilon)\chi_{d_2}(f(\epsilon))\psi_{h_1}(\epsilon)\psi_{h_2}(f(\epsilon))\\
		&   \underset{u_1 \in \mathbb{F}_q}{\sum}\bar{\psi}(Tr(u_1\epsilon)-u_1a) \underset{u_2 \in \mathbb{F}_q}{\sum}\bar{\psi}(Tr(u_2f(\epsilon))-u_2b)
	\end{align*}
	Here we have considered $u_1, u_2 \in \mathbb{F}_q$, such that $\psi(\beta)=\bar{\psi}(u_1\beta)$ and $\psi'(\beta)=\bar{\psi}(u_2\beta)$ for all $\beta \in \mathbb{F}_{q^m}$. Let us denote $\bar{\psi} \circ Tr = \widehat{\psi}$, canonical additive character of $\mathbb{F}_{q^m}$. Then the above equation takes the form
	
	\begin{align*}
		\chi_{f,a,b}(d_1,d_2,h_1,h_2) =& \underset{u_1,u_2 \in \mathbb{F}_q}{\sum}\bar{\psi}(-u_1a-u_2b)\underset{\epsilon\in \mathbb{F}_{q^m}\textbackslash S}{\sum} \chi_{d_1}(\epsilon)\chi_{d_2}(f(\epsilon))\psi_{h_1}(\epsilon)\psi_{h_2}(f(\epsilon))\widehat{\psi}(u_1\epsilon+u_2\epsilon).
	\end{align*}
	
	There exists, $v_1,v_2 \in \mathbb{F}_{q^m}$, such that $\psi_{h_1}(\epsilon)=\widehat{\psi}(v_1\epsilon)$ and $\psi_{h_2}(\epsilon)=\widehat{\psi}(v_2f(\epsilon))$. Moreover we consider $\chi_{d_1}=\chi^{m_1}_{q^m-1}$ and $\chi_{d_2}=\chi^{m_2}_{q^m-1}$ for some multipicative character $\chi_{q^m-1}$ of order $q^m-1$, where $m_1 \in \{ 0,1, \dots ,q^m-2 \}$ and $m_2 = \frac{q^m-1}{d_2}$. Then we get,
	
	\begin{align*}
		\chi_{f,a,b}(d_1,d_2,h_1,h_2) =& \underset{u_1,u_2 \in \mathbb{F}_q}{\sum}\bar{\psi}(-u_1a-u_2b)\underset{\epsilon\in \mathbb{F}_{q^m}\textbackslash S}{\sum} \chi_{q^m-1}(\epsilon^{m_1}f(\epsilon)^{m_2})\widehat{\psi}(v_1\epsilon+v_2f(\epsilon)+u_1\epsilon+u_2f(\epsilon))\\
		=& \underset{u_1,u_2 \in \mathbb{F}_q}{\sum}\bar{\psi}(-u_1a-u_2b)\underset{\epsilon\in \mathbb{F}_{q^m}\textbackslash S}{\sum} \chi_{q^m-1}(F_1(\epsilon))\widehat{\psi}(F_2(\epsilon)),
	\end{align*}
	
	where, 
	\begin{align*}
		F_1(x)=& x^{m_1}f(x)^{m_2},\\
		F_2(x)=& (u_1+v_1)x+(u_2+v_2)f(x). 
	\end{align*}
	\textbf{Case I:} When $F_2(x) \neq r(x)^{p}-r(x) + \beta $ for any $r(x)\in \mathbb{F}_{q^m}(x)$ and $\beta \in \mathbb{F}_{q^m}$ and $F_1(x)=x^{m_1}{f(x)}^{m_2} \neq \beta {(h(x))}^{q^m-1}$ for any $h(x)\in \mathbb{F}_{q^m}(x)$, % where $\mathbb{F}_{q^m}$ is the algebraic closure of $\mathbb{F}_{q^m}$,
	 we can use Lemma \ref{lemma2} to get the following :
	
	$$\big|\chi_{f,a,b}(d_1,d_2,h_1,h_2) \big| \leq q^2(2n+1)q^{m/2}+q^2(|S|-1).$$
	\newline
	\textbf{Case II:} $F_2(x)=h(x)^{p}-h(x)+ \beta$ for some $h(x) \in \mathbb{F}_{q^m}(x) , \beta \in \mathbb{F}_{q^m}$.
	% and $F_1(x)=x^{m_1}{f(x)}^{m_2} \neq{(h(x))}^{q^m-1}$ for any $h(x)\in \mathbb{F}_{q^m}(x)$. %, where $\mathbb{F}$ is the algebraic closure of $\mathbb{F}_{q^m}$.\\
	\\
	\par For $F_2(x)=h(x)^{p}-h(x)+ \beta$ for some $h(x) \in \mathbb{F}_{q^m}(x), \beta \in \mathbb{F}_{q^m}$; we can show that $g_1=g_2=1$.
	For this let us consider $f(x)=c_1{x^j}\frac{f_1}{f_2}$ and $h(x)=c_2{x^k}\frac{h_1}{h_2}$ where $c_1 \in \mathbb{F}^*_{q^m}, c_2 \in \mathbb{F}^*, j,k \in \mathbb{Z}$ and $f_1, f_2 \in \mathbb{F}_{q^m}[x]$ $h_1, h_2 \in \mathbb{F}_{q^m}[x]$ are monic polynomials such that $x$ divides none of them and $\gcd(f_1,f_2)=1, \gcd(h_1,h_2)=1$. Then 
	\begin{align*}
		%F_2(x) =& h(x)^{q^m}-h(x)\\
		(u_1+v_1)x+(u_2+v_2)f(x)=&h(x)^{p}-h(x) +\beta
	\end{align*}
	Setting $(u_1+v_1)=k_1$  and  $(u_2+v_2)=k_2$ we get,
	\begin{align}\label{eq3.1}
		(k_1xf_2+k_2c_1x^jf_1){h_2}^{p} =& (c_2^{p}x^{kp}h_1^{p}-c_2x^kh_1h_2^{p-1}+\beta h_2^p){f_2}
	\end{align}
	%k_1x+k_2{(c_1x^j\frac{f_1}{f_2})}=& {(c_2x^k\frac{h_1}{h_2})}^{q^m}-(c_2x^k\frac{h_1}{h_2})\\
	
	Let $k_2 \neq 0$, then from (\ref{eq3.1}) we get $f_2|h_2^{p}$ since $(f_1,f_2)=1$. Also, since $\gcd(h_2^{p},c_2^{p}x^{k(p)}h_1^{p}-c_2x^kh_1h_2^{p-1}+\beta h_2^p)=1$, it will imply $h_2^{p}|f_2$. This further gives us $f_2=h_2^{p}$, which creates contradiction due to coprimality restrictions i.e., $\gcd(f_1,f_2)=1, \gcd(h_1,h_2)=1$. Therefore we must have $k_2=0$ i.e., $u_2+v_2=0.$
	
	So,
	
	\begin{align}\label{eq3.2}
		k_1x{h_2}^{p} &= c_2^{p}x^{kp}h_1^{p}-c_2x^kh_1h_2^{p-1}+ \beta h_2^p,
	\end{align}
	
	i.e. $h_2$ divides $c_2^{p}x^{kp}h_1^{p}-c_2x^kh_1h_2^{p-1} +\beta h_2^p$. This gives a contradiction since $\gcd(h_2,c_2^{q^m}x^{kq^m}h_1^{q^m}-c_2x^kh_1h_2^{q^m-1})=1$. Therefore $k_1=0$ i.e. $u_1+v_1=0$.
	\par Let us take $\psi_{f_1}$ to be an additive character of $\mathbb{F}_q$-order $f_1$. Then we will have,
	\begin{align*}
		\psi_{f_1}(f_1(\epsilon))=0
		\implies  \widehat{\psi_1}(v_1f_1(\epsilon))=0
		\implies    \widehat{\psi_1}(-u_1f_1(\epsilon))=0
	\end{align*}
	Since $u_1 \in \mathbb{F}_q$ and $\mathbb{F}_q$-order of $\widehat{\psi_1}$ is $x^m-1$, therefore we must have $	x^m-1|(-u_1f_1)$. But we already have, $f_1$ is a $\mathbb{F}_q$-divisor of $x^m-1$. Therefore, we get $u_1=0 $. This further implies $v_1=0$ since $k_1=0$.
	
	Similarly from $k_2=0$ it can be shown that $u_2=v_2=0$. With this we arrive at the conclusion that, if $F_2(x)=h(x)^{p}-h(x)+ \beta$ for some $h(x) \in \mathbb{F}(x)$, where $\mathbb{F}$ is the algebraic closure of $\mathbb{F}_{q^m}$ and $\beta \in \mathbb{F}_{q^m}$, then $g_1=g_2=1$.
	\par Let $F_1(x) = x^{m_1}{f(x)}^{m_2} \neq \beta{(h(x))}^{q^m-1}$ for any $h(x)\in \mathbb{F}_{q^m}(x)$ and $\beta \in \mathbb{F}_{q^m}$. In this case we use Lemma \ref{lemma1} to prove that
	$$\big|\chi_{f,a,b}(d_1,d_2,h_1,h_2) \big| \leq q^2(n)q^{m/2}=nq^{\frac{m}{2}+2}.$$ 
	 In case $F_1(x) = x^{m_1}{f(x)}^{m_2} = \beta{(h(x))}^{q^m-1}$ for some  $h(x)\in \mathbb{F}_{q^m}(x)$ and and $\beta \in \mathbb{F}_{q^m}$, then $f(x) \notin \mathcal{R}_{q^m}^n$ due to violation of first criterion in the definition.\\
	
	\textbf{Case III:} $F_2(x)=h(x)^{p}-h(x)+\beta$ and $F_1(x)=x^{m_1}{f(x)}^{m_2} = \beta{(h(x))}^{q^m-1}$ for some $h(x)\in \mathbb{F}_{q^m}(x)$ and $\beta \in \mathbb{F}_{q^m}$.\\
	\par When $F_1(x)=x^{m_1}{f(x)}^{m_2}=\beta{(h(x))}^{q^m-1}$ for some $h(x)\in \mathbb{F}_{q^m}(x)$ we can prove that $d_1=d_2=1$. We can refer to \cite{sharma2} for the proof. So here we have $$(e_1,e_2,g_1,g_2)=(1,1,1,1).$$
	
	Summarising all the above cases we observe that 
	$$\big|\chi_{f,a,b}(d_1,d_2,h_1,h_2)\big| \leq (2n+1)q^{\frac{m}{2}+2}$$
	where $(d_1,d_2,h_1,h_2) \neq (1,1,1,1)$.
	Now,
	\begin{align*}
		N_{f,a,b}(e_1,e_2,g_1,g_2)  \geq & \frac{\theta(e_1)\theta(e_2)\Theta(g_1)\Theta(g_2)}{q^2} \frac{}{} [(q^m-(2n+1)q^{\frac{m}{2}+2})(W(e_1)W(e_2)W(g_1)W(g_2)-1)]\\
		\geq & \theta(e_1)\theta(e_2)\Theta(g_1)\Theta(g_2)q^{\frac{m}{2}}[q^{\frac{m}{2}-2}-(2n+1)W(e_1)W(e_2)W(g_1)W(g_2)]
	\end{align*}
	To show that there exists an element $\epsilon \in \mathbb{F}_{q^m}$ such that for preassigned $a,b \in \mathbb{F}_q$ and a rational function $f(x) \in \mathcal{R}_{q^m}^n$,$(\epsilon,f(\epsilon))$ is a primitive normal pair with $Tr_{{\mathbb{F}_{q^m}}/{\mathbb{F}_q}}(\epsilon)=a$ and $Tr_{{\mathbb{F}_{q^m}}/{\mathbb{F}_q}}(f(\epsilon))=b$, it is enough to show $N_{f,a,b}(q^m-1,q^m-1,x^m-1,x^m-1)>0$, which is evident if 
	\begin{equation}\label{suff}
		q^{\frac{m}{2}-2}>(2n+1)W(e_1)W(e_2)W(g_1)W(g_2).
	\end{equation}
	
	\section{The prime sieve}
	We will use the following two lemmas to deduce a sieving inequality that will help to improve the sufficient condition (\ref{suff}).
	
	\begin{lemma}\label{sieve1}
		Let $d$ be a divisor of $q^m-1$ and $p_1, p_2,\dots , p_r$ are the remaining distinct primes dividing $q^m-1$. Furthermore, let $g$ be a divisor of $x^m-1$ such that $g_1, g_2, \dots , g_s$ are the remaining distinct irreducible polynomials dividing $x^m-1$. Then
		\begin{align*}
			N_{f,a,b}(q^m-1,q^m-1,x^m-1,x^m-1)\geq \overset{r}{\underset{i=1}{\sum}}N_{f,a,b}(p_id,d,g,g)+  \overset{r}{\underset{i=1}{\sum}}N_{f,a,b}( d,p_i d, g,g) \\
			+ \overset{s}{\underset{i=1}{\sum}}N_{f,a,b}(d, d,g_i g,g)+\overset{s}{\underset{i=1}{\sum}}N_{f,a,b}(d, d,g,g_ig)-(2r+2s-1)N_{f,a,b}(d, d, g, g).
		\end{align*}
	\end{lemma}
	
	\begin{lemma}\label{sieve2}
		Let $d,m,q \in \mathbb{N}, g \in \mathbb{F}_q[x]$ be such that $q$ is a prime power, $m \geq 5, d|q^m-1,$ and $g|x^m-1$. Let $e$ be a prime number which divides $q^m-1$ but not $d$ and $h$ be any irreducible polynomial dividing $x^m-1$ but not $g$. Then we have the following bounds:
		\begin{align*}
			\big| N_{f,a,b}(ed,d,g,g)-\theta(e)N_{f,a,b}(d,d,g,g)\big| \leq &  (2n+1)\theta(e)\theta(d)^2\Theta(g)^2W(d)^2W(g)^2q^{\frac{m}{2}}\\
			\big| N_{f,a,b}(d,ed,g,g)-\theta(e)N_{f,a,b}(d,d,g,g)\big| \leq &
			(2n+1)\theta(e)\theta(d)^2\Theta(g)^2W(d)^2W(g)^2q^{\frac{m}{2}}\\
			\big| N_{f,a,b}(d,d,hg,g)-\Theta(h)N_{f,a,b}(d,d,g,g)\big| \leq & (2n+1)\Theta(h)\theta(d)^2\Theta(g)^2W(d)^2W(g)^2q^{\frac{m}{2}}\\
			\big| N_{f,a,b}(d,d,g,hg)-\Theta(h)N_{f,a,b}(d,d,g,g)\big| \leq & (2n+1)\Theta(h)\theta(d)^2\Theta(g)^2W(d)^2W(g)^2q^{\frac{m}{2}}    
		\end{align*}        
	\end{lemma}
	
	Using these two lemmas \ref{sieve1} and \ref{sieve2} we deduce the following:
	\begin{theorem}\label{psieve}
		Let $d,m,q \in \mathbb{N}, g \in \mathbb{F}_q[x]$ be such that $q$
		is a prime power, $m\geq 5, d|q^m-1$ and $g|x^m-1$. Let $d$ be a divisor of $q^m-1$ and $p_1,p_2,\dots, p_r$ be the remaining distinct primes dividing $q^m-1$. Furthermore let $g$ be a divisor of $x^m-1$ such that $g_1,g_2,\dots,g_s$ are the remaining distinct irreducible factors of $x^m-1$. Define:
		$$l := 1 - 2  \overset{n}{\underset{i=1}{\sum}}\frac{1}{p_i} - \overset{k}{\underset{i=1}{\sum}} \frac{1}{q^{{\mathrm deg}(g_i)}}, l>0$$
		and
		$$ \Lambda := \frac{2(r+s)-1}{l}+2.$$
		Then $N_{f,a,b}(q^m-1,q^m-1,x^m-1,x^m-1)>0$ if
		\begin{align}\label{primesieve}
			q^{\frac{m}{2}-2}>(2n+1)W(d)^2W(g)^2L
		\end{align}
	\end{theorem}
	
	\begin{proof}
		For our convenience, let us denote $ N_{f,a,b}(q^m-1,q^m-1,x^m-1,x^m-1)$ by $N$. Then from Lemma \ref{sieve1} and \ref{sieve2}, we have 
		\begin{align*}
			N \geq &\overset{r}{\underset{i=1}{\sum}}\{N_{f,a,b}(p_id,d,g,g)- \theta(p_i)N_{f,a,b}(d,d,g,g)\} +			   
			\overset{r}{\underset{i=1}{\sum}}\{N_{f,a,b}( d,p_i d, g,g) -\theta(p_i)N_{f,a,b}(d,d,g,g)\}\\
			+& \overset{k}{\underset{i=1}{\sum}}\{N_{f,a,b}(d, d,g_i g,g)- \Theta(g_i)N_{f,a,b}(d,d,g,g)\}
			+\overset{k}{\underset{i=1}{\sum}}\{N_{f,a,b}(d, d,g,g_ig)-\Theta(g_i)N_{f,a,b}(d,d,g,g)\}\\
			+& \{ 2\overset{r}{\underset{i=1}{\sum}}\theta (p_i)+2\overset{s}{\underset{i=1}{\sum}}
			\Theta (g_i)\}N_{f,a,b}(d, d, g, g)
			-(2r+2s-1)N_{f,a,b}(d, d, g, g).\\
			=& \overset{r}{\underset{i=1}{\sum}}\{N_{f,a,b}(p_id,d,g,g)- \theta(p_i)N_{f,a,b}(d,d,g,g)\} 
			+  \overset{r}{\underset{i=1}{\sum}}\{N_{f,a,b}( d,p_i d, g,g) -\theta(p_i)N_{f,a,b}(d,d,g,g)\}\\
			+& \overset{k}{\underset{i=1}{\sum}}\{N_{f,a,b}(d, d,g_i g,g)- \Theta(g_i)N_{f,a,b}(d,d,g,g)\}
			+\overset{k}{\underset{i=1}{\sum}}\{N_{f,a,b}(d, d,g,g_ig)-\Theta(g_i)N_{f,a,b}(d,d,g,g)\}\\ + &lN_{f,a,b}(d,d,g,g)\\
			\geq& \theta(d)^2\Theta(g)^2l[\{ -(2n+1)W(d)^2W(g)^2q^{\frac{m}{2}}L\}+\{q^{m-2}-(n+1)+(2n+1)q^{\frac{m}{2}}\}]
		\end{align*}
		\\
		Therefore, when $l>0$, then $N_{f,a,b}(q^m-1,q^m-1,x^m-1,x^m-1)>0$ if 
		\begin{align}
			q^{\frac{m}{2}-2}>(2n+1)W(d)^2W(g)^2L.
		\end{align}

	\end{proof}

	\section{Numerical computations}
	In this section, we particularly study the case when finite fields are of characteristic $5$ and rational functions of degree sum $4$ and try to determine pairs $(q,m) \in \mathcal{A}_5^4$. For large calculations we use sagemath \cite{sage}.
	\par The following lemma will be used extensively throughout the process.
	\begin{theorem}
		Let $\nu>0$ be a real number and $t$ be a positive integer.Then $W(t)<D.(t)^{\frac{1}{\nu}}$, where $D=\frac{2^r}{(p_1p_2\dots p_r)^{\frac{1}{\nu}}}$ and $p_1,p_2,\dots,p_r$ are primes $\leq 2^\nu$ that divide $t$.
	\end{theorem}
	
	\textbf{Remark:} For large values of $q$ and $m$, calculation of  $W(q^m-1)$ is very time consuming. So using the above lemma we find an upper bound to the term $W(q^m-1)$ as follows:
	$$W(q^m-1) \leq D.(q^m-1)^{\frac{1}{\nu}}<D.(q^m)^{\frac{1}{\nu}}.$$
	Also since $x^m-1$ can have atmost $m$ linear factors in $\mathbb{F}_q[x]$. So, $$W(x^m-1)\leq 2^m.$$
	Using these two observations we rewrite our sufficient condition (\ref{suff}) as:
	\begin{equation}\label{D}
		q^{\frac{m}{2}-2}>(2n+1)D^2q^{\frac{2m}{\nu}}2^{2m}.
	\end{equation}
	In our case, $n=4,m\geq 5, q=5^k$ for some positive integer $k$.
	We choose $\nu=21.57$. Then we get $D\leq 1.52 \times10^{4906}$. For these values (\ref{D}) holds for $k\geq 385897$ and $m\geq 5$.
	\par Now for $3\leq k\leq 385897$, we present a table that shows suitable values of $\nu$ and an integer $m_k$ corresponding to each value of $k$ such that for all $m \geq m_k$ and that particular $\nu$, (\ref{D}) holds.
	
	\begin{table}
		\begin{center}
			\begin{tabular}{|c|c|c|}
				\hline
				$\nu$ &$k$ &$m_k$ \\
				\hline
				13.7 &1379,1380, \dots ,385896 &6\\
				11.3 & 212, 213, \dots, 1378 &7\\
				10.1 & 84, 85, \dots , 211 &8\\
				9.52 & 48,49, \dots , 83 &9\\
				9.2 & 33,34, \dots,47 &10\\
				8.9 & 26,27,\dots , 32 &11\\
				8.7 & 21,22, \dots , 25 &12\\
				8.6 &18,19,20 &13\\
				8.5 &16,17 &14\\
				8.5 &14,15 &15\\
				8.4 & 13 & 16\\
				8.4& 12& 17\\
				8.5&11&18\\
				8.4 & 10&19\\
				8.4 & 9&21\\
				8.5 & 8&24\\
				8.4 & 7&28\\
				8.5& 6&36\\
				8.8 &5&52\\
				9.4&4&99\\
				11.3&3&520\\
				\hline
			\end{tabular}
			\caption{Values of $\nu$ and $m_k$ corresponding to $k$.\label{table1}}
		\end{center}
	\end{table}

	Analysing the values of $m_k$ and corresponding range of $k$ from Table \ref{table1}, we observe that (\ref{D}) holds for all $k\geq 48$ and $m\geq 9$. Based on this we claim the following:
	
	\begin{lemma}\label{kg3}
		For $k\geq 3$ and $m\geq 9$, $(q,m)\in \mathcal{A}_5^4$. 
	\end{lemma}
	\begin{proof}
		Since for $3\leq k\leq 47$, (\ref{D}) holds for all $m\geq m_k$, so for $9\leq k<m_k$, we first check the following inequality,
		\begin{equation}\label{D1}
			q^{\frac{m}{2}-2}>(2n+1)D^2q^{\frac{2m}{\nu}}(W(x^m-1))^2.
		\end{equation}
		This inequality reduces the number of possible exceptional pairs to $220$. For all these possibilities, we can find suitable values of $d$ and $g$ ( as defined in Theorem \ref{psieve}), so that (\ref{primesieve}) holds.
	\end{proof}
	
	We study the pairs for $k=1,2$ separately. Let us consider $m=m'5^j$ for $j\geq0$ such that $\gcd(5,m')=1$. Then we can divide our discussion into the following two cases:\\
	(i)  $m'| q^2-1$,\\
	(ii)  $m' \nmid q^2-1$.
	\newline
	
	\textbf{Case I}: First let us consider the pairs $(q,m)$, where $m'| q^2-1$ and $k=1$.
	\par Then possible values of $m' $ are $1$, $2$, $3$, $4$, $6$, $8$, $12$, $24$. Since 
	$W(x^m-1)=W(x^ {m'}-1)$, So we can rewrite (\ref{D}) as 
	\begin{equation}\label{m'nk1}
		q^{\frac{m'5^j}{2}-2}>(2n+1)D^2q^{\frac{2m'5^j}{\nu}}2^{2m'}.
	\end{equation}
	Then (\ref{m'nk1}) holds when,\\
	(i) $m'=1$ and $j\geq 4$,\\
	(ii) $m'=2$ ,$3$, $4$, $6$ and $j\geq 3$,\\
	(iii) $m'=8$ ,$12$, $24$ and $j\geq 2$.\\
	Therefore $(5,m)\in \mathcal{A}_5^4$ unless $m$ $=$ $10$, $12$, $15$, $20$, $24$, $25$, $30$, $40$, $50$, $60$, $75$, $100$, $120$, $125$, $150$.
	For these pairs we directly check sufficient condition (\ref{suff}) and reduce the list of possible exceptional values of $m$ to $10$, $12$, $15$, $20$, $24$, $30$. Again for these values of $m$, we try to find suitable $d$ and $g$ ( as defined in Theorem \ref{psieve}) for which (\ref{primesieve}) is satisfied ( see Table \ref{table2}). But for the pairs $(5,10)$, $(5,12)$, $(5,20)$, $(5,24)$; we could not find such $d$ and $g$. So in this case, these are the possible exceptional pairs.
	\par Now let $k=2$ and $m' \nmid q^2-1$. \\
	Then possible values of $m'$ are $1$, $2$, $3$, $4$, $6$, $8$, $12$, $13$, $16$, $24$, $26$, $39$, $48$, $52$, $76$, $104$, $156$, $208$, $312$, $624$. We find out that (\ref{m'nk1}) holds when,
	(i) $m'=1$ ,$2$, $3$ and $j\geq 3$,\\
	(i) $m'=4$, $6$, $8$, $12$, $13$, $16$ and $j\geq 2$,\\
	(i) $m'=24$, $26$, $39$, $48$, $52$, $76$, $104$, $156$, $208$, $312$, $624$ and $j\geq 1$.\\
	Therefore $(25,m) \in \mathcal{A}_5^4$ unless $m=10$, $12$, $13$, $15$, $16$, $20$, $24$, $25$, $26$, $30$, $39$, $40$, $48$, $50$, $52$, $60$, $65$, $75$, $76$, $80$, $104$, $156$, $208$, $312$, $624$. We reduce theses exceptions by removing the pairs $(25,m)$ that satisfies (\ref{suff}). This leaves us with possible exceptional values of $m=10,12,13,15,16,24,48.$ For these pairs we try to find $d$ and $g$(as defined in Theorem (\ref{psieve}) such that (\ref{primesieve}) holds( see Table \ref{table2}). Since we fail to find such $d$ and $g$ for $m=10,12,15$, so $(25,10),(25,12),(24,25)$  adds up to the list of exceptional pairs.
	
	\begin{table}
		\begin{center}
			\begin{tabular}{|c|c|c|c|c|c|c|c|}
				\hline
				SL No. &$(q,m)$ &$d$ &$r$ &$g$ &$s$ &$l$ &$L$\\ %&$lhs$ &$rhs$\\
				\hline
				1&(5,13) &2 &1 &$(x+4)$ &3 &0.990 &9.07\\ %&1397.54 &1305.77\\
				2&(5,15) &2 &5 &$(x+4)$ &1 &0.633 &19.37\\ % &6987.71 &2789.19\\
				3&(5,21) &2 & 4&$(x+4)$ &4 &0.850 &19.65 \\ %&873464.05 &2830.02\\
				4&(5,22) &6 &4 &$(x+1)$ &5 &0.480 &37.04 \\%&$1.95\times 10^6$  &21541.286\\
				5&(5,30) &462 &7 &$(x+1)$ &3 &0.298 &65.71 \\%&$1.22\times 10^9$ &605579.54\\
				6&(5,80) &66 &12 &$(x+1)(x+2)(x+3)(x+4)$ &4 &0.466 &68.50 \\%&$3.63\times 10^{26}$ &$1.01\times 10^7$\\
				7&$(5^2,11)$ &6 &4 &$(x+4)$ &2 &0.883 &14.46 \\%&78125.00 &8329.05\\
				8&$(5^2,13)$ &6 &3 &$(x+4)$ &6 &0.980 &19.34\\% &$1.95\times 10^6$ &11140.11\\
				9&$(5^2,14)$ &6 &5 &$(x+1)$ &5 &0.692 &29.45\\% &$9.76\times10^6$ &16964.30\\
				10&$(5^2,15)$ &6 &9 &$(x+4)$ &2 &0.231 &93.03 \\%&$4.88\times 10^6$ &53584.68\\
				11&$(5^2,18)$ &42 &8 &$(x+1)$ &9 &0.216 &155.00\\% &$6.10\times10^9$ &357128.95\\
				12&$(5^2,21)$ &6 &10 &$(x+4)$ &8 &0.348 &102.67 \\%&$7.63\times10^{11}$ &59137.07\\
				13&$(5^2,36)$ &546 &12 &$(x+1)(x+2)(x+3)(x+4)$ &16 &0.095 &579.62\\% &$2.32\times 10^{22}$ &$3.42\times10^8$\\
				14&$(5^2,48)$ &9282 &10 &$g'$ &22 &0.066 &716.32 \\%&$5.68\times 10^30$ &$1.77\times10^{15}$\\
				\hline
			\end{tabular}
			
			\caption{Pairs $(q, m)$ when $k=1,2$, for which Theorem \ref{primesieve} holds for the above choices of $d$, $r$, $g$ and $s$.\label{table2}}
			[Here $g'=(x + 1)(x + 2)(x + 3)(x + 4)(x + \beta)(x + \beta + 1) (x + \beta + 2)(x + \beta + 3)(x + \beta + 4)(x + 2\beta)(x + 2\beta + 1)(x + 2\beta + 2)(x + 2\beta + 3)(x + 2\beta + 4),$ with $\beta$ as an generating element of the field.]
			
		\end{center}
		
	\end{table}
	
	Before moving to the second case let us recall the following lemmas.
	
	\begin{lemma}[\cite{cohen3}]\label{delbound}
		Suppose $q=5^k, m'>4$ and $\Bar{m}=\gcd(q-1, m')$. If $M'$ denotes the number of distinct irreducible factors of $x^{m'} -1$ over $\mathbb{F}_q$, $e$ is the order of $q(\mod {m'})$ and $\delta(q,m)= \frac{M'}{e}$; Then the  following holds.
		\begin{itemize}
			\item $\delta(q,m)\leq \frac{1}{2}$, for $m=2\Bar{m}$,
			\item $\delta(q,m)\leq \frac{3}{8}$, for $m= 4\Bar{m}$,
			\item $\delta(q,m)\leq \frac{13}{36}$, for $m= 6\Bar{m}$,
			\item $\delta(q,m)\leq \frac{1}{3}$, otherwise.
		\end{itemize}
		
	\end{lemma}

	\begin{lemma}[\cite{sharma2},Lemma 5.3]
		Assume that $q=p^k,k \in \mathbb{N}$ and $m=m'p^j, j\geq 0$ is a positive integer such that $m'\nmid q-1$ and $\gcd(m',p)=1$.  Let $e (>2)$ denote the order of $q \mod m'$. Then $d=q^m-1$ and $g$ as the product of the irreducible factors of $x^{m'}-1$ of degree less than $e$. Then, in the notation of Lemma  $\ref{primesieve}$, we have $L\leq 2m'$.
		
	\end{lemma}
	Since $m'\nmid q^2-1$ implies $m'\nmid q-1$ and $e>2$. We can transform (\ref{D}) using the above two  lemmas as follows:
	
	\begin{equation*}
		q^{\frac{m'5^j}{2}-2}>9D^2q^{\frac{2m'5^j}{\nu}}2^{2m'\delta(q,m')}2m',
	\end{equation*}
	which is also true if 
	\begin{equation}\label{del}
		q^{\frac{m}{2}-2}>18D^2q^{\frac{2m}{\nu}}2^{2m\delta(q,m')}m.
	\end{equation}
	Now we can proceed to the second case.
	
	\textbf{Case II:} $m' \nmid q^2-1$
	Let $k=1$. Then possible values of $m'$ are $5$, $7$, $9$, $10$, $11$, $13$, $14$, $15$, $16$, $17$, $18$, $19$, $20$, $21$, $22$, $23$. Out of these only $16$ is there such that $16=4.\Bar{m}$. So for others we can rewrite (\ref{del}) as
	\begin{equation}\label{del}
		5^{\frac{m}{2}-2}>18D^2m5^{\frac{2m}{\nu}}2^{\frac{2m}{3}}.
	\end{equation}
	For $\nu =11.3$, (\ref{del}) holds for$m\geq 1566$. So we list out the possible exceptional values of $m$ in this case to be $m= 9$, $10$, $11$, $13$, $14$, $15$, $17$, $18$, $20$, $21$, $22$, $23$, $25$, $35$, $45$, $50$, $55$, $65$, $70$, $75$, $85$, $90$, $95$,  $100$, $105$, $110$, $115$, $125$, $175$, $225$, $250$, $275$, $325$, $350$, $375$, $425$, $450$, $475$, $500$, $525$, $550$, $575$, $625$, $875$, $1125$, $1250$, $1375$. Out of them $m=9$, $10$, $11$, $13$, $14$, $15$, $18$, $20$, $21$, $22$ don't satisfy (\ref{suff}). For these values we try to calculate $d$ and $g$ so that condition (\ref{primesieve}) is satisfied ( see Table \ref{table2}). For $m=9$, $10$, $11$, $14$, $18$ and $20$, we fail to do so. Therefore they also belong to the list of exceptional values of $m.$ 
	\par The case $m'=16$ is handled separately. For this (\ref{m'nk1})
	takes the form 
	\begin{equation}\label{del2}
		5^{\frac{m}{2}-2}>18D^2m5^{\frac{2m}{\nu}}2^{\frac{3m}{4}}.
	\end{equation}
	For $\nu =11.3$, (\ref{del2}) holds for $m\geq 342$. This leaves us with $m=16$, $80$. But for $m=16$, neither (\ref{suff})
	holds nor we could find any $d$ and $g$ to satisfy (\ref{primesieve}). So $(5,16)$ is also a possible exceptional pair. With this we exhaust the case for $k=1.$
	\par For $k=2$, there is no $m'$ such that $m'=2\Bar{m},4\Bar{m} $ or $6\Bar{m}$. So here we can rewrite (\ref{m'nk1}) as 
	\begin{equation}\label{del3}
		(25)^{\frac{m}{2}-2}>18D^2m{(25)}^{\frac{2m}{\nu}}2^{\frac{2m}{3}}.
	\end{equation}
	With $\nu =11.3$, (\ref{del3}) holds for all $m\geq 159$.\\
	Therefore we need to check for $m \in \{ 9,10,..,158\} \backslash \{12,13,16,24,26,39,48,52,78,104,156\}.$
	Out of these, for $m=9,$ $10$, $11$, $14$, $15$, $18$, $21$ and $36$, (\ref{suff}) do not hold. So we search for $d$ and $g$ so that (\ref{primesieve}) holds for them. But for $m=9$, $10$, we couldnot
	find such $d$ and $g$, so $(25,9)$, $(25,10)$ are added up to the list of possible exceptional pairs.
	
	Summarizing the discussions in this section we obtain the proof of Theorem \ref{main}.

	\section{Declarations}
	\textbf{Conflict of interest} The authors declare no competing interests.
	
	\textbf{Ethical Approval} Not applicable.
	
	\textbf{Funding} The first author is supported by NFOBC fellowship (NBCFDC Ref. No. 231610154828).
	
	\textbf{Data availability} Not applicable.

\end{document}